%% file: ReconstructionTautv3.tex
\documentclass[11pt,twoside, reqno]{amsart}
\input{header.tex}

\begin{document}

\title[Reconstruction from tautological images on Hilbert schemes]{Some ways to reconstruct a sheaf from its tautological image on a Hilbert scheme of points}
\author[A.\ Krug and J.\ V.\ Rennemo]{Andreas Krug and J\o rgen Vold Rennemo}
\begin{abstract}
For $X$ a smooth quasi-projective variety and $X^{[n]}$ its associated Hilbert scheme of $n$ points, we study two canonical Fourier--Mukai transforms $\D(X)\to \D(X^{[n]})$, the one along the structure sheaf and the one along the ideal sheaf of the universal family. For $\dim X\ge 2$, we prove that both functors admit a left-inverse. This means in particular that both functors are faithful and injective on isomorphism classes of objects. Using another method, we also show in the case of an elliptic curve that the Fourier--Mukai transform along the structure sheaf of the universal family is faithful and injective on isomorphism classes. Furthermore, we prove that the universal family of $X^{[n]}$ is always flat over $X$, which implies that the Fourier--Mukai transform along its structure sheaf maps coherent sheaves to coherent sheaves.      
\end{abstract}
 \maketitle

\section{Introduction}

Tautological bundles on Hilbert schemes of points have been intensively studied from various perspectives, and have been used for many applications, starting in the 1960s when they were studied on symmetric products of curves; see \cite{Schwarzenberger--bundlesplane, Schwarzenberger--secant, Mattuck--sym}. They are defined by means of the Fourier--Mukai transform along the universal family of the Hilbert scheme. More concretely, let $X$ be a smooth quasi-projective variety over $\IC$, let $X^{[n]}$ be the associated Hilbert scheme of $n$ points, and let $\Xi_n\subset X^{[n]}\times X$ be the universal family of length $n$ subschemes of $X$, together with its projections $X\xleftarrow q \Xi_n\xrightarrow p X^{[n]}$. Given a rank $r$ vector bundle $E$ on $X$, there is the associated tautological bundle $E^{[n]}=p_*q^*E$ of rank $rn$ on $X^{[n]}$.

A very natural question is whether the bundle on $X$ can be reconstructed from its associated tautological bundle on the Hilbert scheme:
\begin{align}\label{eq:Q}
 E^{[n]}\cong F^{[n]}\quad \xRightarrow{\quad \textbf{\large ?}\quad}\quad E\cong F\,.
\end{align}
This question was studied quite recently by Biswas and Parameswaran \cite{Biswas-Parameswaran--reconstructioncurves} and by Biswas and Nagaraj \cite{Biswas-Nagaraj--reconstructioncurves, Biswas-Nagaraj--reconstructionsurfaces}.
Maybe surprisingly, Question \eqref{eq:Q} has a negative answer if $X=\IP^1$; see \cite[Sect.\ 2.1]{Biswas-Nagaraj--reconstructionsurfaces}. On the other hand, the answer to Question \eqref{eq:Q} is affirmative for semi-stable vector bundles on curves of genus $g(X)\ge 2$, see \cite{Biswas-Parameswaran--reconstructioncurves, Biswas-Nagaraj--reconstructioncurves}, and for arbitrary vector bundles on surfaces; see \cite{Biswas-Nagaraj--reconstructionsurfaces}.

In the present paper, we generalise, strengthen, and complement these results in various directions. 
We extend the results from vector bundles to coherent sheaves and, more generally, objects in the derived category. We obtain new affirmative answers to Question \eqref{eq:Q} for varieties of higher dimension and  
for elliptic curves. In most cases, we prove something slightly stronger than an affirmative answer to Question \eqref{eq:Q}, namely the existence of left-inverses to the functor $(\_)^{[n]}$. Furthermore, we obtain similar results if we replace $(\_)^{[n]}$, which is the Fourier--Mukai transform along the structure sheaf of the universal family, by the Fourier--Mukai transform along the ideal sheaf of the universal family.

Let us describe the results in more detail. As alluded to above, given a smooth quasi-projective variety $X$ over $\IC$ and a positive integer $n$, we consider the Fourier--Mukai transform
\[
(\_)^{[n]}=\FM_{\reg_{\Xi_n}}\cong p_*\circ q^*\colon \D(X)\to \D(X^{[n]})
\]
between the derived categories of perfect complexes. We can now extend Question \eqref{eq:Q} from vector bundles to objects of the derived category $\D(X)$. In other words, we ask ourselves whether the functor $(\_)^{[n]}\colon \D(X)\to \D(X^{[n]})$ is injective on isomorphism classes. However, for a better understanding of this functor, we first prove the following, which extends a result of Scala \cite[Sect.\ 2.1]{Sca2} from surfaces to varieties of arbitrary dimension.
\begin{theorem}[\autoref{prop:flatalgebraic}]
The universal family $\Xi_n\subset X^{[n]}\times X$ is always flat over $X$. Hence, the functor $(\_)^{[n]}$ sends coherent sheaves on $X$ to coherent sheaves on $X^{[n]}$. 
\end{theorem}
Besides $(\_)^{[n]}$, there is a second, equally canonical, Fourier--Mukai transform to the Hilbert scheme, namely the one along the universal ideal sheaf
\[
 \FF_n=\FM_{\cI_{\Xi_n}}\cong \pr_{X^{[n]}*}\bigl(\cI_{\Xi_n}\otimes \pr_X^{*}(\_)\bigr)\colon \D(X)\to \D(X^{[n]})\,.
\]
This functor was studied in the case that $X$ is a K3 surface in \cite{Add} and \cite{MMdefK3}, and in the case that $X$ is a surface with $\Ho^1(\reg_X)=0=\Ho^2(\reg_X)$ in \cite{KSEnriques}. In the case of a K3 surface, it was shown that $\FF_n$ is a $\IP^{n-1}$-functor. In the case of a surface with $\Ho^1(\reg_X)=0=\Ho^2(\reg_X)$, it was shown that $\FF_n$ is fully faithful. In both cases, it follows that $\FF_n$ is injective on isomorphism classes and faithful. 

Let us now summarise the results of this paper concerning reconstruction of objects from their images under the functors $(\_)^{[n]}$ and $\FF_n$. 

\begin{theorem}[\autoref{thm:lefti}]\label{thm:Psiintro}
If $X$ is a smooth quasi-projective variety of dimension $d=\dim X\ge 2$, there are left-inverses of both functors $(\_)^{[n]},\FF_n\colon \D(X)\to \D(X^{[n]})$ for every $n\in \IN$. In particular, $(\_)^{[n]}$ and $\FF_n$ are faithful, and for every pair $E,F\in \D(X)$ we have
\[E^{[n]}\cong F^{[n]}\quad \Longleftrightarrow \quad E\cong F \quad \Longleftrightarrow \quad \FF_n(E)\cong \FF_n(F)\,.\]   
\end{theorem}
\autoref{thm:Psiintro} is proved in \autoref{sect:smallstratum}. We need two different constructions of a left-inverse of the functors $(\_)^{[n]}$ and $\FF_n$. Which one of the two constructions works depends on whether or not $n$ is of the form $\binom{m+d}d$ for some $m\in \IN$.    

If $n$ is not of the form $\binom{m+d}d$ for some $m\in \IN$, we write $n=\binom{m+d-1}d +\ell$ for some $0< \ell< \binom{m+d-1}{d-1}=\rank(\Sym^m \Omega_X)$. Let $\IG\subset X^{[n]}$ be the locus of subschemes $\xi\subset X$ such that
\[\Spec(\reg_{X,x}/\fm_x^m)\subset \xi\subset \Spec(\reg_{X,x}/\fm_x^{m+1})\,\]
for some $x \in X$; in particular these are \emph{punctual} subschemes in the sense that $\supp(\xi) = \{x\}$.
The morphism $f\colon \IG\to X$, $\xi\mapsto x$, given by forgetting the scheme structure of $\xi$, identifies $\IG$ with the Grassmannian bundle $\Gr(\Sym^m \Omega_X, \ell)$ of rank $\ell$ quotients of $\Sym^m\Omega_X$. We denote the universal quotient bundle on $\IG$ by $\cQ$ and write the closed embedding as $\iota\colon \IG\hookrightarrow X^{[n]}$.
In \autoref{thm:Gres}, we prove that the functor $K:=f_*(\cQ^\vee\otimes \iota^*(\_))\colon \D(X^{[n]})\to \D(X)$ is a left-inverse of $(\_)^{[n]}$ and of $\FF_n[1]$.

If $n=\binom{m+d}d$ for some $m\in \IN$ and $d\ge 2$, we construct another functor $N\colon \D(X^{[n]})\to \D(X)$ which is then proven to be a left-inverse of $(\_)^{[n]}$ and of $\FF_n[1]$ in \autoref{thm:Pres}. The functor $N$ is somewhat similar to the functor $K$ described above, but, instead of $\IG\subset X^{[n]}$, it uses a certain locus of pairs of punctual schemes inside the nested Hilbert scheme $X^{[n-1,n]}\subset X^{[n-1]}\times X^{[n]}$; see \autoref{sect:nested} for details. Obviously, \autoref{thm:Gres} and \autoref{thm:Pres} together proof \autoref{thm:Psiintro}.

In \autoref{sect:symstack}, using equivariant sheaves on the cartesian product $X^n$, we also give two other reconstruction methods. These methods give an affirmative answer to question \eqref{eq:Q} in many cases; namely for arbitrary objects in $\D(X)$ if $X$ is a surface, and for reflexive sheaves if $X$ is projective of dimension $d>2$. However, they are slightly weaker than the methods of \autoref{sect:smallstratum} which work for every object in $\D(X)$ for every $d\ge 2$. The reason that we decided to include the constructions of \autoref{sect:symstack} in the paper is twofold. Firstly, the proofs in \autoref{sect:symstack} are somewhat easier than those in \autoref{sect:smallstratum}. Secondly, the construction of \autoref{subsect:reflexive} is used in \cite{BiswasKrug} to prove an analogous reconstruction result for Hitchin pairs. 

If the variety $X$ has fixed-point free automorphisms, we find a further reconstruction method which, contrary to the other methods, also works for curves. Here, we only state the most significant consequence, and refer to \autoref{sect:freeautos} for details. 

\begin{theorem}[\autoref{cor:abelian}]\label{thm:ellipticintro}
Let $X$ be an elliptic curve. Then, for every $n\in \IN$ and every pair $E,F\in \D(X)$, we have  
\[E^{[n]}\cong F^{[n]}\quad \Longleftrightarrow \quad E\cong F\,.\]  
\end{theorem}

In summary, we now are now not very far away from a complete answer of Question \eqref{eq:Q}, in its generalised form for objects of the derived category. 

\autoref{thm:Psiintro} gives a complete affirmative answer for $\dim X\ge 2$. 
In the curve case, the picture is a bit more subtle.
Note first that on a curve $X$, every object $E \in D(X)$ is the direct sum of its shifted cohomology sheaves, and so by exactness of $(-)^{[n]}$ (see \autoref{prop:flatalgebraic}) it is enough to answer Question (\ref{eq:Q}) under the assumption $E, F \in \Coh(X)$.
Furthermore, the torsion part of a sheaf $E$ is easily recovered from $E^{[n]}$, and so the general form of Question (\ref{eq:Q}) reduces to the special form where $E$ and $F$ are vector bundles.

For $X=\IP^1$, as mentioned above, the answer to Question (\ref{eq:Q}) is negative. Note however, that the answer is affirmative for line bundles on $\IP^1$; see \autoref{rem:jetreconstruction}. For curves of genus $1$, we have an affirmative answer for arbitrary objects of the derived category by \autoref{thm:ellipticintro}. For curves of genus $g(X)\ge 2$, we have an affirmative answer for semi-stable vector bundles by \cite{Biswas-Nagaraj--reconstructioncurves}, and also for a slightly bigger class of vector bundles, namely those where the Harder--Narasimhan factors have slopes contained in a sufficiently small interval; see \cite[Prop.\ 2.1]{Biswas-Nagaraj--reconstructionsurfaces}.     

The remaining open question is thus:
\begin{question}
Do there exist pairs of unstable sheaves $E$ and $F$ on a curve of genus $g\ge 2$ with $E\not \cong F$ but $E^{[n]}\cong F^{[n]}$ for some $n\ge 2$?    
\end{question}

\subsection*{General conventions}

All our schemes and varieties are defined over the complex numbers $\IC$.
Given two schemes $X$ and $Y$, we write the projections from their product to the factors as $\pr_X\colon X\times Y\to X$ and $\pr_Y\colon X\times Y\to Y$. If $Z\subset X\times Y$ is a subscheme, we denote the restrictions of the projections to $Z$ by $\pr^Z_X$ and $\pr^Z_Y$, respectively.

For $X$ a scheme, let $\D(\QCoh(X))$ be the derived category of quasi-coherent $\reg_X$-modules. 
We write $\D(X):=\Perf(X)\subset \D(\QCoh(X))$ for the full subcategory of perfect complexes, i.e.\ those complexes which are locally quasi-isomorphic to bounded complexes of vector bundles.  
If $X$ is a smooth variety, $\D(X)$ is equivalent to the bounded derived category of coherent sheaves. We use the same notation for derived functors as for their non-derived versions. For example, given a morphism $f\colon X\to Y$, we write $f^*\colon \D(Y)\to \D(X)$ instead of $Lf^*\colon \D(Y)\to \D(X)$.

We say that a functor $F\colon \cC\to \cD$ is \textit{injective on isomorphism classes} if for all pairs of objects $C,C'\in \cC$, we have that $F(C)\cong F(C')$ implies that $C\cong C'$. 
A \textit{left-inverse} of a given functor $F\colon \cC\to \cD$ is a functor $G\colon \cD\to \cC$ such that $G\circ F\cong \id_{\cC}$ (often in the literature, this is called a \textit{quasi left-inverse} functor as we require the composition only to be isomorphic, not equal, to the identity). A functor admitting a left-inverse is injective on isomorphism classes and faithful. 

In this paper, $\IN$ denotes the set of positive integers, and $\IN_0$ denotes the set of non-negative integers. 

\subsection*{Acknowledgements}
The authors thank Ben Anthes, Pieter Belmans, John Christian Ottem and S\"onke Rollenske for helpful discussions and comments.

\section{Hilbert schemes of points and Fourier--Mukai transforms}

\subsection{Hilbert schemes of points and symmetric quotients}

From now on, $X$ will always be a smooth quasi-projective variety over $\IC$.
Given a non-negative integer $n$, the \textit{Hilbert scheme $X^{[n]}$ of $n$ points on $X$} is the fine moduli space of zero-dimensional closed subschemes of $X$ of length $n$. It is smooth if and only if $\dim X\le 2$ or $n\le 3$; see \cite{Fog, Cheah--cellularHilb}. 

We consider the cartesian product $X^n$ together with the action of the symmetric group $\sym_n$ given by permutation of the factors. The quotient $X^{(n)}:=X^n/\sym_n$ by that action is called the \textit{$n$-th symmetric product} of $X$. We denote the quotient morphism by $\pi\colon X^n\to X^{(n)}$ and write the points of the symmetric product in the form $x_1+\dots+x_n:=\pi(x_1,\dots,x_n)$. 

There is the \textit{Hilbert--Chow morphism}
\[
 \mu\colon X^{[n]}\to X^{(n)}\quad,\quad [\xi]\mapsto \sum_{x\in \xi}\ell(\reg_{\xi,x})\cdot x
\]
sending a length $n$ subscheme to its weighted support. If $X$ is a curve, the Hilbert--Chow morphism is an isomorphism. 

\subsection{Flatness of the universal family}

Being a fine moduli space, the Hilbert scheme $X^{[n]}$ comes equipped with a universal family $\Xi_n=\Xi\subset X^{[n]}\times X$ which is flat and finite of degree $n$ over $X^{[n]}$. In fact, it is also flat over $X$, as we show in the following.

\begin{theorem}\label{prop:flatalgebraic}
 For every smooth variety $X$ and every $n\in \IN$, the universal family $\Xi_n\subset X^{[n]}\times X$ is flat over $X$. 
\end{theorem}
\begin{proof}
By GAGA, a morphism of schemes of finite type over $\IC$ is flat if and only if its analytification is; see \cite[Expos\'e XII, Prop.\ 3.1]{SGAI}. Note that the analytification $\Xi^{\an}$ of $\Xi$ is the universal family of the \textit{Douady space} of $X^{\an}$, that means the moduli space of zero-dimensional analytic subspaces of $X^{\an}$ of length $n$; see \cite{Douady}. Hence, we can deduce 
\autoref{prop:flatalgebraic} from the analogous result in the category of complex spaces, which is \autoref{prop:flatanalytic} below.  
\end{proof}
\begin{theorem}\label{prop:flatanalytic}
For every complex manifold $M$ and every $n\in \IN$, the universal family $\Xi_n\subset M^{[n]}\times M$ of the Douady space $M^{[n]}$ is flat over $M$. 
\end{theorem}
\begin{proof}
We first prove the assertion in the case that $M=\IC^d$. By generic flatness, there is a non-empty open subset $U\subset M$ such that the restriction of $\Xi_n$ is flat over $U$. Let $(\xi,x)\in \Xi_n\subset M^{[n]}\times M$ with $x\not\in U$. The action of $\Aut(\IC^d)$ is transitive. Hence, there exists an $\phi\in \Aut(\IC^d)$ with $\phi(x)\in U$. Let $\phi^{[n]}$ denote the induced automorphism of $M^{[n]}$. As $\phi^{[n]}\times \phi$ is an automorphism of $\Xi_n$, the flatness of $\reg_{\Xi_n,(\xi,x)}$ over $\reg_{M,x}$ follows from the flatness of $\reg_{\Xi_n,(\phi^{[n]}(\xi),\phi(x))}$ over $\reg_{M,\phi(x)}$. 

Let now $M$ be an arbitrary complex manifold of dimension $d$, and let $(\xi,x)\in \Xi_n$. Write $x_0=x$ and $\mu(\xi)=n_0\cdot x+n_1\cdot x_1+\dots +n_tx_t$ where $\mu\colon M^{[n]}\to M^{(n)}$ denotes the Douady--Barlet morphism, which is the analytic analogue of the Hilbert--Chow morphism; see \cite{Barlet}. Now, choose pairwise disjoint open neighbourhoods $U_0, \dots, U_k$ of $x_0, \dots x_k$ such that every $U_i$ is isomorphic to an open subset of $\IC^d$. Then, $U_0^{[n_0]}\times U_1^{[n_1]}\times \dots \times U_k^{[n_k]}$ is an open neighbourhood of $\xi\in M^{[n]}$. Hence, $\Xi^{U_0}_{n_0}\times U_1^{[n_1]}\times \dots \times U_k^{[n_k]}$, where $\Xi^{U_0}_{n_0}\subset U_0^{[n_0]}\times U_0$ denotes the universal family of $U_0^{[n_0]}$, is an open neighbourhood of $(\xi, x)$ in $\Xi_n$. The restriction of the projection $\pr^{\Xi_n}_X\colon \Xi_n\to X$ to this open neighbourhood is given by the composition
\[
\Xi^{U_0}_{n_0}\times U_1^{[n_1]}\times \dots \times U_k^{[n_k]}\to \Xi^{U_0}_{n_0}\to U_0\hookrightarrow M\,.  
\]
The first morphism is the projection to the first factor, hence flat. The second map is the projection from the universal family. As $U_0$ is an open subset of $\IC^d$, for which we already proved the assertion, this morphism is flat too. The third morphism is the open embedding, hence flat. It follows that the whole composition is flat, which is what we needed to show.  
\end{proof}

\subsection{Canonical Fourier--Mukai transforms}\label{subsect:tautological}

Given a smooth quasi-projective variety $X$ and a positive integer $n$, we define the \textit{tautological functor} 
\[
 (\_)^{[n]}:=\pr^{\Xi}_{X^{[n]}*}\circ \pr^{\Xi*}_{X}\colon \D(X)\to \D(X^{[n]})\,.
\]
The functor is well-defined in that it preserves perfect complexes. Indeed, every pull-back preserves perfect complexes, and the push-forward $\pr^{\Xi}_{X^{[n]}*}$ preserves perfect complexes as well since $\pr^{\Xi}_X$ is flat and finite. 
Note that $\pr^{\Xi}_{X^{[n]}*}$ and $\pr^{\Xi*}_{X}$ are both already exact on the level of the abelian categories of coherent sheaves, before deriving, since $\pr^{\Xi}_{X^{[n]}}$ is finite and $\pr^{\Xi}_{X}$ is flat by \autoref{prop:flatalgebraic}. 
Hence, $(\_)^{[n]}$ restricts to an exact functor
\[
 (\_)^{[n]}:=\pr^{\Xi}_{X^{[n]}*}\circ \pr^{\Xi*}_{X}\colon \Coh(X)\to \Coh(X^{[n]})\,;
\]
see \cite[Sect.\ 2.1]{Sca2} for a different proof of this fact in the case that $X$ is a surface.
Let $E\in \Coh(X)$ be a vector bundle of rank $r$.
The fact that $\pr^{\Xi}_{X^{[n]}}$ is flat and finite of degree $n$ implies that $E^{[n]}$ is a vector bundle of rank $rn$. 
By the projection formula, the tautological functor can be identified with the Fourier--Mukai transform 
\[
 (\_)^{[n]}\cong \FM_{\reg_{\Xi}}=\pr^{X\times X^{[n]}}_{X^{[n]}*}\bigl(\reg_\Xi\otimes \pr^{X\times X^{[n]}*}_{X}(\_)\bigr)
\]
along the structure sheaf of the universal family. Another much-studied functor is the Fourier--Mukai transform
\[
 \FF_n:=\FM_{\cI_\Xi}=\pr^{X\times X^{[n]}}_{X^{[n]}*}\bigl(\cI_\Xi\otimes \pr^{X\times X^{[n]}*}_{X}(\_)\bigr)
\]
along the ideal sheaf of the universal family. We only consider the functor $\FF_n$ in the case that $X$ is projective since only in this case does it preserve perfect complexes. To see this, we note that the short exact sequence 
\[
 0\to \cI_\Xi\to \reg_{X^{[n]}\times X}\to \reg_\Xi\to 0
\]
induces, for every perfect complex $E\in \D(X)$, an exact triangle
\begin{align*}
 \FM_{\cI_\Xi}(E)\to \FM_{\reg_{X^{[n]}\times X}}(E)\to \FM_{\reg_\Xi}(E)\to \FM_{\cI_\Xi}(E)[1] 
\end{align*}
consisting, a priori, of objects of $\D(\QCoh(X^{[n]}))$. We have already seen that $E^{[n]}\cong \FM_{\reg_\Xi}(E)$ is perfect. Furthermore, we have $\FM_{\reg_{X^{[n]}\times X}}(E)\cong \reg_{X^{[n]}}\otimes_{\IC}\Ho^*(E)$. This is again a perfect complex since we assume $X$ to be projective which implies that $\Ho^*(E)$ is a finite dimensional graded vector space. Since the subcategory $\D(X^{[n]})\subset \D(\QCoh(X^{[n]}))$ of perfect complexes is triangulated, it follows that $\FF_n(E)\cong \FM_{\cI_\Xi}(E)$ is perfect too. Hence, we have a well-defined functor $\FF_n\colon \D(X)\to \D(X^{[n]})$.

\begin{lemma}\label{lem:tautbasechange}
Let $T$ be a scheme, $\cZ\subset T\times X$ a flat family of length $n$ subschemes of $X$, and $\psi\colon T\to X^{[n]}$ the classifying morphism for $\cZ$. Then we have an isomorphism of functors
\[
\psi^*\circ (\_)^{[n]}\cong\pr^{\cZ}_{T*}\circ \pr^{\cZ*}_X\cong \FM_{\reg_{\cZ}} \,.
\]
\end{lemma}
\begin{proof}
This follows from base change along this cartesian diagram with flat vertical arrows:
\[
\begin{gathered}[b]
\xymatrix{
 \cZ \ar^{\psi\times \id_X}[r]  \ar_{\pr_T^{\cZ}}[d] & \Xi\ar^{\pr^\Xi_{X^{[n]}}}[d] \ar^{\pr^\Xi_X}[r]& X   \\
 \cZ \ar^{\psi}[r] & X^{[n]} \,. & 
}\\[-\dp\strutbox]
\end{gathered}
\qedhere
\]

\end{proof}

\section{Reconstruction using the symmetric product}\label{sect:symstack}

\subsection{Reconstruction for reflexive sheaves}\label{subsect:reflexive}

The following construction, which reconstructs reflexive sheaves on a smooth quasi-projective variety $X$ of dimension $d\ge 2$ from their associated tautological sheaves on $X^{[n]}$, is inspired by \cite[Sect.\ 1]{Stapletontaut}.  

Let $X^{[n]}_0\subset X^{[n]}$ be the open locus of reduced subschemes. Let $X^n_0\subset X^n$ the open complement of the big diagonal, and let $X_0^{(n)}=\pi(X_0^n)\subset X^{(n)}$. This means
\begin{align*}
 X^n_0&=\bigl\{(x_1,\dots,x_n)\in X^n\mid \text{ the $x_i$ are pairwise distinct } \bigr\}\,,\\ 
 X^{(n)}_0&=\bigl\{x_1+\dots+x_n\in X^{(n)}\mid \text{ the $x_i$ are pairwise distinct } \bigr\}\,. 
\end{align*}
The Hilbert--Chow morphism induces an isomorphism $X_0^{[n]}\xrightarrow \sim X_0^{(n)}$. Let $j\colon X_0^{(n)}\hookrightarrow X^{[n]}$ be the open embedding induced by the inverse of that isomorphism. For $i=1,\dots, n$, let $\pr_i\colon X^n\to X$ denote the projection to the $i$-th factor, and let $\pr_i^0\colon X_0^n\to X$ be the restriction of that projection.  

\begin{lemma}\label{lem:openpullback}
We have an isomorphism of functors \[\pi_0^*\circ j^*\circ(\_)^{[n]} \cong \bigoplus_{i=1}^n\pr_i^{0*}\colon \D(X)\to \D(X_0^n)\,.\]
\end{lemma}

\begin{proof}
This is stated in \cite[Lem.\ 1.1]{Stapletontaut} in the case that $E$ is a vector bundle and $X$ is a surface. The proof of our more general statement is exactly the same. For convenience, we quickly reproduce the proof in slightly different words.

Note that $X_0^{(n)}$ is the fine moduli space of reduced length $n$ subschemes on $X$, and the quotient $\pi_0\colon X_0^n\to X_0^{(n)}$ is the classifying morphism for the family $Z=\bigsqcup_{i=1}^n\Gamma_{i}\subset X_0^n\times X$ given by the disjoint union of the graphs $\Gamma_i$ of the projections $\pr^0_i\colon X_0^n\to X$. Hence, by \autoref{lem:tautbasechange}, we have $\pi_0^*\circ j^*\circ(\_)^{[n]}\cong \FM_{\reg_Z}$. Since $\reg_Z\cong \bigoplus_{i=1}^n\reg_{\Gamma_{i}}$, and $\FM_{\reg_{\Gamma_{i}}}\cong \pr_i^{0*}$, we get the asserted isomorphism. 
\end{proof}

\begin{theorem}\label{thm:reflexive}
Let $X$ be a smooth projective variety of dimension $d\ge 2$. Then, for every pair $E,F\in\Refl(X)\subset \Coh(X)$ of reflexive sheaves, we have 
\[
 E^{[n]}\cong F^{[n]}\quad \Longrightarrow \quad E\cong F\,.
\] 
\end{theorem}

\begin{proof}
Let $E^{[n]}\cong F^{[n]}$, and write $\alpha\colon X^n_0\hookrightarrow X^n$ for the open embedding. Then, by \autoref{lem:openpullback}, 
\begin{align}\label{eq:alphaequal}
\alpha^*(\bigoplus_{i=1}^n\pr_i^*E)\cong \bigoplus_{i=1}^n\pr_i^{0*}E\cong \pi_0^*j^*E^{[n]}\cong \pi_0^*j^*F^{[n]} \cong \bigoplus_{i=1}^n\pr_i^{0*}F\cong \alpha^*(\bigoplus_{i=1}^n\pr_i^*F)\,.   
\end{align}
The sheaves $\bigoplus_{i=1}^n\pr_i^*E$ and $\bigoplus_{i=1}^n\pr_i^*F$ are reflexive since flat pull-backs preserve reflexivity; see \cite[Prop.\ 1.8]{Hartshorne--stable}. Note that the codimension of the complement of $X^n_0$ in $X^n$ is $d\ge 2$. Hence, for a reflexive sheaf $\cE$ on $X^n$, we have $\cE\cong \alpha_*\alpha^* \cE$; see \cite[Prop.\ 1.11]{Hartshorne--generalized}. Combining this with 
\eqref{eq:alphaequal} gives
\begin{align}\label{eq:Cequal}
\bigoplus_{i=1}^n\pr_i^*E\cong  \alpha_*\alpha^*(\bigoplus_{i=1}^n\pr_i^*E) \cong  \alpha_*\alpha^*(\bigoplus_{i=1}^n\pr_i^*F)\cong  \bigoplus_{i=1}^n\pr_i^*F\,.
\end{align}
Let $\delta\colon X\hookrightarrow X^n$ be the embedding of the small diagonal. We have $\delta^*\pr_i^*\cong \id$ for every $i=1,\dots,n$. Combining this with \eqref{eq:Cequal} gives 
\[
E^{\oplus n}\cong \delta^*(\bigoplus_{i=1}^n\pr_i^*E)\cong \delta^*(\bigoplus_{i=1}^n\pr_i^*F)\cong F^{\oplus n}\,.
\]
The category $\Coh(X)$ is Krull-Schmidt; see \cite{Atiyah--Krull-Schmidt}. Hence, $E^{\oplus n}\cong F^{\oplus n}$ implies $E\cong F$.
\end{proof}

\subsection{Equivariant sheaves and the McKay correspondence}

We quickly collect some facts about equivariant sheaves, their derived categories, and functors between them for later use. For further details, the reader may consult, among others, \cite[Sect.\ 4]{BKR}, \cite{Elagin}, or \cite[Sect.\ 2.2]{Krug--remarksMcKay}.
Let $Y$ be a smooth variety equipped with an action of a finite group $G$. 
A \textit{$G$-equivariant sheaf} on $Y$ is a pair $(F,\lambda)$ consisting of a 
coherent sheaf $F\in \Coh(Y)$ and a \textit{$G$-linearisation} $\lambda$, which means a family $\{\lambda_g\colon F\xrightarrow\sim g^*F\}_{g\in G}$ of isomorphisms such that for every pair $g,h\in G$ the following diagram commutes:
\[
\xymatrix{
 F \ar^{\lambda_g}[r]  \ar@/_6mm/^{\lambda_{hg}}[rrr] & g^*F \ar^{g^*\lambda_h}[r] & g^*h^*F\ar^{\cong}[r] & (hg)^*F\,.  
} 
\]
We obtain the abelian category $\Coh_G(Y)$ of equivariant sheaves, where a morphism between two equivariant sheaves is a morphism of the underlying coherent sheaves that commutes with the linearisations. We denote the bounded derived category of $\Coh_G(Y)$ by $\D_G(Y)$. 

Let $Z$ be a second smooth variety equipped with a $G$-action, and let $f\colon Y\to Z$ be a $G$-equivariant morphism. Then, if $(E,\lambda)\in \Coh(Z)$ is a $G$-equivariant sheaf, the pull-back $f^*E$ canonically inherits a $G$-linearisation induced by $\lambda$. This gives a functor $f^*\colon \Coh_G(Z)\to \Coh_G(Y)$ together with its derived version $f^*\colon \D_G(Z)\to \D_G(Y)$. If $f$ is proper, we also get a push-forward $f_*\colon \D_G(Y)\to \D_G(Z)$.  

If $H\le G$ is a subgroup, we have an exact functor $\Res_G^H\colon \Coh_G(Y)\to\Coh_H(Y)$ given by restricting the linearisations to $H$. This functor is exact, so it induces a functor on the derived categories $\Res_G^H\colon \D_G(Y)\to \D_H(Y)$.

If the $G$-action on $Y$ is trivial, a $G$-linearisation of a sheaf $E\in \Coh(Y)$ is the same as a $G$-action on $E$, i.e.\ a compatible family of $G$-actions on $E(U)$ for every open subset $U\subset Y$. Hence, we can take the invariants of an equivariant sheaf over every open subset, which gives a functor $(\_)^G\colon \Coh_G(Y)\to \Coh(Y)$. Since we are working over $\IC$, this functor is exact and we get an induced functor $(\_)^G\colon \D_G(Y)\to \D(Y)$.

Let now $X^n$ be the $n$-fold cartesian product of a smooth quasi-projective surface $X$, and let $G=\sym_n$ be the symmetric group acting by permutation of the factors. We consider the reduced fibre product
\begin{align*}
\xymatrix{
 (X^{[n]}\times_{X^{(n)}} X^n)_{\mathsf{red}} \ar^{\quad\quad p}[r]  \ar^{q}[d] & X^n\ar^{\pi}[d]   \\
 X^{[n]} \ar^{\mu}[r] & X^{(n)}\,. 
}
\end{align*}
In this set-up, the derived McKay correspondence of Bridgeland--King--Reid \cite{BKR} and Haiman \cite{Hai} states that the functor $\Phi:=p_*\circ q^*\colon \D(X^{[n]})\to \D_{\sym_n}(X^n)$ is an equivalence. One can deduce quite easily that 
\[\Psi:=(\_)^{\sym_n}\circ q_*\circ p^*\colon \D_{\sym_n}(X^n)\to \D(X^{[n]})\]
is an equivalence too, but it is not the inverse of $\Phi$; see \cite[Prop.\ 2.8]{Krug--remarksMcKay}.

For use in the next subsection, we define a functor
\[
\CC\colon \Coh(X)\to \Coh_{\sym_n}(X^n)\quad, \quad E\mapsto \CC(E)=(\bigoplus_{i=1}^n\pr_i^*E, \lambda) 
\]
where $\pr_i\colon X^n\to X$ is the projection to the $i$-th factor, and $\lambda_g$ is, for $g\in \sym_n$, the direct sum of the canonical isomorphisms $\pr_i^*E\xrightarrow\sim g^*\pr_{g(i)}^*E$. As the projections $\pr_i$ are flat, the functor $\CC\colon \Coh(X)\to \Coh_{\sym_n}(X^n)$ is exact, hence it induces a functor
$\CC\colon \D(X)\to \D_{\sym_n}(X^n)$. 
\subsection{Reconstruction of complexes using the McKay correspondence}

In this subsection, with the exception of \autoref{rem:reflinverse}, let $X$ be a smooth quasi-projective surface.
In this case, we can strengthen \autoref{thm:reflexive} by constructing a left-inverse to the functor $(\_)^{[n]}\colon \D(X)\to\D(X^{[n]})$ on the level of derived categories. The key is the following result, which can be seen as a refinement of \autoref{lem:openpullback}.
\begin{theorem}[{\cite[Thm.\ 3.6]{Krug--remarksMcKay}}]\label{thm:PsiC}
$\Psi^{-1}\circ (\_)^{[n]}\cong \CC$. 
\end{theorem}

\begin{theorem}\label{thm:Psiinverse}
The functor $G=(\_)^{\sym_n}\circ \delta^*\circ \Psi^{-1}\colon \D(X^{[n]})\to \D(X)$ is left inverse to $(\_)^{[n]}$. 
\end{theorem}
\begin{proof}
By \autoref{thm:PsiC}, we have $G\circ (\_)^{[n]}\cong (\_)^{\sym_n}\circ \delta^*\circ\CC$. The only factor of the composition $(\_)^{\sym_n}\circ \delta^*\circ\CC$ that is not already exact on the level of (equivariant) coherent sheaves is $\delta^*$. Hence, for a given $E\in \D(X)$, we can compute $(\delta^*\CC(E))^{\sym_n}$ by replacing $E$ by a resolution by vector bundles and then applying the non-derived functors. Accordingly, it suffices to prove the isomorphism of functors $(\_)^{\sym_n}\circ \delta^*\circ\CC\cong \id$ on the category $\VB(X)$ of vector bundles on $X$. For every $E\in \VB(X)$, since $\pr_i\circ \delta=\id_X$, we have $\delta^*\CC(E)\cong E^{\oplus n}$ with the $\sym_n$-action on a local section $s=(s_1,\dots s_n)\in E^{\oplus n}(U)$ given by $g\cdot s =(s_{g^{-1}(1)}, \dots, s_{g^{-1}(n)})$. Hence, $s$ is $\sym_n$-invariant if and only if it is of the form $s=(t,\dots, t)$ for some $t\in E(U)$. It follows that the projection $E^{\oplus n}\to E$ to any factor induces a functorial isomorphism $(\delta^*\CC(E))^{\sym_n}\cong E$.   
\end{proof}
\begin{remark}\label{rem:reflinverse}
Using $\sym_n$-equivariant sheaves, we can also slightly strengthen \autoref{thm:reflexive} by constructing, for $X$ of arbitrary dimension, a concrete left-inverse functor of the restriction $(\_)^{[n]}\colon \Refl(X)\to \Coh(X^{[n]})$ of the tautological functor to the category of reflexive sheaves. First, we note that the pull-back of a coherent sheaf along the quotient $\pi_0\colon X_0^n\to X_0^{(n)}$ is naturally equipped with a $\sym_n$-linearisation given by the canonical isomorphisms $\pi_0^*\xrightarrow\sim g^*\circ \pi_0^*$. Hence, we can regard $\pi_0^*$ as a functor $\Coh(X_0^{(n)})\to \Coh_{\sym_n}(X_0^n)$ instead of a functor $\Coh(X_0^{(n)})\to \Coh(X_0^n)$. Doing this, we can see that, instead of the statement of \autoref{lem:openpullback}, we get the isomorphism of functors $\pi_0^*\circ j^*(\_)^{[n]}\cong \alpha^*\circ\CC\colon \D(X)\to \D_{\sym_n}(X^n_0)$. Now, setting \[H:=(\_)^{\sym_n}\circ \delta^*\circ \alpha_*\circ \pi_0^*\circ i^*\colon \Coh(X^{[n]})\to \Coh(X)\,,\] and combining the proofs of \autoref{thm:reflexive} and \autoref{thm:Psiinverse}, we get $H\circ (\_)^{[n]}\cong \id$ as an isomorphism of endofunctors of $\Refl(X)$.     
\end{remark}
Recall that, as explained in \autoref{subsect:tautological}, we have for $E\in \D(X)$ a natural exact triangle
\begin{align}\label{eq:FFF}
E^{[n]}[-1] \to \FF_n(E)\to \reg_{X^{[n]}}\otimes \Ho^*(E)\to E^{[n]}\,.  
\end{align}
Hence, if $G(\reg_{X^{[n]}})=0$, we would have an isomorphism of functors $G\circ (\_)^{[n]}[-1]\cong G\circ \FF_n$, and accordingly $\FF_n$ also had a left-inverse, namely $G[1]$. However, this is not the case. One can show easily that 
\begin{align}\label{eq:PsiO}
 \Psi^{-1}(\reg_{X^{[n]}})\cong \reg_{X^n} 
\end{align}
where $\reg_{X^n}$ is equipped with the canonical $\sym_n$-linearisation given by the push-forwards of functions $\reg_{X^n}\xrightarrow\sim g^*\reg_{X^n}$; see \cite[Rem.\ 3.10]{Krug--remarksMcKay}. It follows that $G(\reg_{X^{[n]}})\cong \reg_X$.

In the following, we will adapt the functor $G\colon \D(X^{[n]})\to \D(X)$ in such a way that it annihilates $\reg_{X^{[n]}}$ while still being left-inverse to $(\_)^{[n]}$. This new functor will then (up to shift) also be a left inverse of $\FF_n$.

Let $\sym_k$ act on some variety $Y$. We denote by $\MM_{\alt_k}\colon \Coh_{\sym_k}(Y)\to \Coh_{\sym_k}(Y)$ the tensor product with the sign character, which means that   
$\MM_{\alt_k}(E, \lambda)=(E,\bar\lambda)$ with $\bar\lambda_g=\sgn(g)\cdot \lambda_g$.
This functor is exact, hence induces a functor $\MM_{\alt_k}\colon \D_{\sym_k}(Y)\to \D_{\sym_k}(Y)$.
\begin{theorem}\label{thm:Psiinversecomplicated}
The functor $I=(\_)^{\sym_2}\circ \MM_{\alt_2}\circ \Res_{\sym_n}^{\sym_2}\circ \delta^*\circ \Psi^{-1}\colon \D(X^{[n]})\to \D(X)$ is left-inverse to $(\_)^{[n]}$, and $I[1]$ is left inverse to $\FF_n$.
\end{theorem}
\begin{proof}
By \autoref{thm:PsiC}, we have $I\circ (\_)^{[n]}\cong (\_)^{\sym_2}\circ \MM_{\alt_2}\circ \Res_{\sym_n}^{\sym_2}\circ \delta^*\circ\CC$. Hence, exactly as in the proof of \autoref{thm:Psiinverse}, we can reduce the assertion $I\circ (\_)^{[n]}\cong \id_{\D(X)}$ to the construction of an isomorphism $(\_)^{\sym_2}\circ \MM_{\alt_2}\circ \Res_{\sym_n}^{\sym_2}\circ \delta^*\circ\CC\cong \id$ of endofunctors of $\VB(X)$. 

For $E\in \VB(X)$, we have $\MM_{\alt_2}\Res_{\sym_n}^{\sym_2}\delta^*\CC(E)\cong E^{\oplus n}$ with the $\sym_2$-action on a local section $s=(s_1,\dots s_n)\in E^{\oplus n}(U)$ given by $(1\,\,2)\cdot s =(-s_2,-s_1,-s_3, \dots, -s_{n})$. Hence, $s$ is $\sym_2$-invariant if and only if it is of the form $s=(t,-t, 0,\dots, 0)$ for some $t\in E(U)$. It follows that the projection $E^{\oplus n}\to E$ to any of the first two factors induces a functorial isomorphism $\bigr(\MM_{\alt_2}\Res_{\sym_n}^{\sym_2}\delta^*\CC(E)\bigl)^{\sym_2}\cong E$.

Following the discussion above, for the second assertion it suffices to show $I(\reg_{X^{[n]}})\cong 0$.
By \eqref{eq:PsiO}, we only need to check that $\bigl(\MM_{\alt_2}\Res_{\sym_n}^{\sym_2}\delta^*(\reg_{X^{n}})\bigr)^{\sym_2}=0$. This is the case since the $\sym_2$-action on $\MM_{\alt_2}\Res_{\sym_n}^{\sym_2}\delta^*(\reg_{X^{n}})\cong \reg_X$ is given by multiplication by $\sgn$. 
\end{proof}

\section{Reconstruction using a small stratum of punctual subschemes}\label{sect:smallstratum}

\subsection{Jet bundles}

Let $\Delta\subset X\times X$ be the diagonal and $\cI_\Delta$ its ideal sheaf. For $m\in \IN$, we write the subscheme defined by $\cI_\Delta^m$ as $m\Delta\subset X\times X$. For $E\in \D(X)$ and $m\in \IN_0$, the associated \textit{$m$-jet object} is defined as $\Jet^m E=\FM_{\reg_{(m+1)\Delta}}(E)$. In particular, we have $\Jet^0 E\cong E$. For $m>0$, there is a short exact sequence
\begin{align}\label{eq:sesdiag}
 0\to \delta_*(\Sym^{m}\Omega_X)\to \reg_{(m+1)\Delta}\to \reg_{m\Delta}\to 0\,.
\end{align}
which induces the exact triangle
\begin{align}\label{eq:jettriangle}
 \Sym^{m}\Omega_X\otimes E\to \Jet^{m}E\to \Jet^{m-1}E\to \Sym^{m}\Omega_X\otimes E[1]\,. 
\end{align}
It follows inductively that, if $E\in \VB(X)$ is a vector bundle of rank $r$, the associated $m$-jet object $\Jet^{m} E$ is a vector bundle of rank $\binom{m+d}d r$ where $d=\dim X$. Note that $\Jet^m\reg_X$ has fibres $(\Jet^m\reg_X)(x)=\reg_{X,x}/\fm_x^{m+1}$.

\subsection{The locus of punctual subschemes with Hilbert function concentrated in minimal degrees}\label{subsec:family}

For $0\le \ell\le {\binom{m+d-1}{d-1}}=\rank(\Sym^{m}\Omega_X)$, there is a family of punctual length $n:=\binom{m+d-1}d + \ell$ subschemes of $X$ over $\IG:=\Gr(\Sym^{m}\Omega_X, \ell)$, the Grassmannian of rank $\ell$ quotients of the symmetric product of the cotangent bundle, constructed as follows. Let $f\colon\IG\to X$ be the natural projection, and let $\eps\colon f^*\Sym^{m}\Omega_X\to \cQ$ be the universal quotient bundle.
Using base change along the cartesian diagram
\begin{align*}
\xymatrix{
 \IG \ar@^{(->}^{(\id_{\IG},f)\quad}[r]  \ar^{f}[d] & \IG\times X\ar^{f\times \id_X}[d]   \\
 X \ar@^{(->}^{\delta\quad}[r] & X\times X\,, 
}
\end{align*}
we see that the pull-back of the short exact sequence \eqref{eq:sesdiag} along $f\times \id_X$ is of the form 
\begin{align}\label{eq:sesgraphf}
 0\to (\id_\IG, f)_* f^*\Sym^{m}\Omega_X\to (f\times \id_X)^*\reg_{(m+1)\Delta}\to (f\times \id_X)^*\reg_{m\Delta}\to 0\,.
\end{align}
Modding out the first two terms of \eqref{eq:sesgraphf} by the kernel of the push-forward
\[
  (\id, f)_*\eps\colon (\id_\IG, f)_* f^*\Sym^{m}\Omega_X\to (\id_\IG,f)_*\cQ
\]
of the universal quotient, we get a short exact sequence 
\begin{align}\label{eq:cF}
0\to (\id_\IG, f)_*\cQ \to \cF\to (f\times \id_X)^*\reg_{m\Delta}\to 0
\end{align}
where $\cF$ is a quotient of $(f\times \id_X)^*\reg_{(m+1)\Delta}$ hence, in particular, of       
$(f\times \id_X)^*\reg_{X\times X}\cong \reg_{\IG\times X}$. Thus, $\cF\cong \reg_\cZ$ for some closed subscheme $\cZ\subset \IG\times X$ supported on $\Gamma_f=(f\times\pr_X)^{-1}\Delta$. Note that $\pr_{\IG*}(\id_\IG, f)_*\cQ\cong \cQ$, and by flat base change 
$\pr_{\IG*}(f\times \id_X)^*\reg_{m\Delta}\cong f^*(\Jet^{m-1}\reg_X)$. Hence, by \eqref{eq:cF}, $\pr_{\IG*}\reg_\cZ$ is locally free of rank 
\[
\rank(\pr_{\IG*}\reg_\cZ)=\rank \cQ + \rank (\Jet^{m-1}\reg_X)= \ell + \binom{m+d-1}d =n  \,,
\]
which means that $\cZ$ is a family of length $n$ subschemes of $X$, flat over $\IG$.
We denote the classifying morphism for $\cZ\subset \IG\times X$ by $\iota\colon \IG\to X^{[n]}$. 

Let us note some facts about $\iota\colon \IG\to X$, though they are not logically necessary for the proofs in the following subsection. The classifying morphism $\iota\colon \IG\to X^{[n]}$ is a closed embedding; compare \cite[Sect.\ 2.1]{Goettschebook}. Its image is exactly the locus of length $n$ subschemes $\xi\subset X$ with $\supp\xi =\{x\}$ for some point $x\in X$, which means that they are \textit{punctual}, satisfying $\Spec(\reg_{X,x}/\fm_x^m)\subset \xi\subset \Spec(\reg_{X,x}/\fm_x^{m+1})$. The last property is equivalent to the Hilbert function of $\xi$ being $(1,d, \binom{d+1}{d-1},\dots, \binom{d+m-2}{d-1},\ell, 0,0,\dots)$. For more information on the strata of $X^{[n]}$ parametrising punctual subschemes with given Hilbert functions, see \cite[Sect.\ 2.1]{Goettschebook} and \cite{Iarrobino--memoirs}.       

\subsection{A left-inverse functor using restriction to the Grassmannian bundle}

In this subsection, we construct a left-inverse of $(\_)^{[n]}$ and $\FF_n$ whenever $n$ is not of the form 
$n=\binom{u+d}d$ for any $u\in \IN_0$. Note that, for $d=1$, every positive integer is of the form $\binom{u+d}d$. This means that we do not get left-inverses in the case that $X$ is a curve by our method; see 
\autoref{rem:jetreconstruction} for some further details on the curve case.

\begin{theorem}\label{thm:Gres}
Assume that $n\in \IN$ is not of the form $n=\binom{u+d}k$ for any $u\in \IN_0$. 
Set $m-1:=\max\{u\mid \binom{u+d}d<n\}$ and $\ell:=n-\binom{m+d-1}d$. Let $f\colon \IG=\Gr(\Sym^{m}\Omega_X,\ell)\to X$ be the Grassmannian bundle together with the family $\cZ\subset \IG\times X$ of length $n$ subschemes constructed in \autoref{subsec:family}, and let $\iota\colon \IG\hookrightarrow X^{[n]}$ be the classifying morphism for $\cZ$. Then \[K:=f_*\bigl(\cQ^\vee\otimes\iota^*(\_)\bigr)\colon \D(X^{[n]})\to \D(X)\,,\] where $f^*\Sym^{m}\Omega_X\to \cQ$ is the universal quotient bundle, is left-inverse to the functor $(\_)^{[n]}$. Also, 
$K[1]$ is left-inverse to $\FF_n$.
\end{theorem}
For the proof, we need the following 
\begin{lemma}\label{lem:Grasssod}
Let $Y$ be a smooth variety, $V\in \VB(Y)$, and $0<\ell<\rank V$. Let $f\colon \Gr(V,\ell)\to Y$ be the Grassmannian bundle of rank $\ell$ quotients of $V$ with universal quotient bundle $\cQ$.
Then we have the following isomorphisms of functors:
\begin{enumerate}
 \item $f_*\bigl(\cQ^\vee\otimes f^*(\_)\bigr)\cong0$,
 \item $f_*\bigl(\sHom(\cQ,\cQ)\otimes f^*(\_)\bigr)\cong\id_{\D(Y)}$,
 \item $f_*\circ f^*\cong \id_{\D(Y)}$. 
\end{enumerate}
\end{lemma}
\begin{proof}
By \cite[Thm.\ 2]{Samokhin--hom}, Kapranov's full exceptional collection on a Grassmannian over a point \cite[Sect.\ 3]{Kapranov--homogeneous} has a relative version in the form of a semi-orthogonal decomposition of $\D(\Gr(V,\ell))$, with all of its factors being of the form $\sfS^\alpha(\cQ) \otimes f^*\D(X)$ for some Schur functors $\sfS^\alpha$; for an overview of the theory of semi-orthogonal decompositions see, for example, \cite{Kuz--ICM}. In particular, two of the factors of the semi-orthogonal decomposition are $f^*\D(X)$ and $\cQ \otimes f^*\D(X)$, which means that the functors $f^*\colon \D(X)\to \D(\Gr(V,\ell))$ and $\cQ\otimes f^*(\_)\colon \D(X)\to \D(\Gr(V,\ell))$ are both fully faithful. Their right adjoints are given by $f_*$ and $f_*(\cQ^\vee\otimes \_)$, respectively. The composition of a fully faithful functor with its right adjoint is the identity, which gives (ii) and (iii). 

For (i), we note that the factor $f^*\D(X)$  stands in our semi-orthogonal decomposition on the left of $\cQ\otimes f^*\D(X)$, which means that $\Hom\bigl(\cQ\otimes f^*\D(X), f^*\D(X)\bigr)=0$; compare \cite[Lem.\ 3.2a)]{Kapranov--homogeneous}. Using the right-adjoint $f_*(\cQ^\vee\otimes\_)$ of the functor $\cQ\otimes f^*(\_)$, this Hom-vanishing translates to (i).   
\end{proof}

\begin{proof}[Proof of \autoref{thm:Gres}]
Let $E\in \D(X)$. By \autoref{lem:tautbasechange}, we have a natural isomorphism $\iota^*E^{[n]}\cong \FM_{\reg_\cZ}(E)$. The short exact sequence \eqref{eq:cF} induces the exact triangle 
\begin{align}\label{eq:sometriangle}
 \FM_{(\id_\IG,f)_*\cQ}(E)\to \FM_{\reg_\cZ}(E)\to \FM_{(f\times \id_X)^*\reg_{m\Delta}}(E)\to \FM_{(\id_\IG,f)_*\cQ}(E)[1]\,. 
\end{align}
By the projection formula, we have $\FM_{(\id_\IG,f)_*\cQ}(E)\cong \cQ\otimes f^*E$. By flat base change, we have $\FM_{(f\times \id_X)^*\reg_{m\Delta}}(E)\cong f^*\FM_{\reg_{m\Delta}}(E)\cong f^*\Jet^{m-1} E$. Hence,
\eqref{eq:sometriangle} can be rewritten as 
\[
\cQ\otimes f^*E\to \iota^*E^{[n]}\to f^*\Jet^{m-1} E\to \cQ\otimes f^*E[1]\,. 
\]
Applying $f_*(\cQ^\vee \otimes \_)$, we get the exact triangle
\[
f_*(\sHom(\cQ,\cQ)\otimes f^*E)\to K(E^{[n]})\to f_*(\cQ^\vee\otimes f^*\Jet^{m-1} E)\to \cQ\otimes f_*(\sHom(\cQ,\cQ)\otimes f^*E)[1]\,. 
\]
By \autoref{lem:Grasssod}, we see that $f_*(\cQ^\vee\otimes f^*\Jet^{m-1} E)$ vanishes, and we get a natural isomorphism $K(E^{[n]})\cong f_*(\sHom(\cQ,\cQ)\otimes f^*E)\cong E$.

Now, in order to prove that $K\circ \FF_n\cong [-1]$, because of \eqref{eq:FFF}, we only need to check that $K(\reg_{X^{[n]}})=0$. As $\iota^*\reg_{X^{[n]}}\cong \reg_\IG\cong f^*\reg_X$, this follows directly from \autoref{lem:Grasssod}(i). 
\end{proof}

\begin{remark}
In the special case $d=2=n$, the functor $K$ of \autoref{thm:Gres} coincides with the functor $I$ of \autoref{thm:Psiinversecomplicated}. To see this, note that, in this special case, the subvariety $\iota(\IG)$ agrees with the whole boundary divisor of $X^{[2]}$, i.e.\ the locus of non-reduced length 2 schemes, and $f\colon \IG\to X$ is a $\IP^1$-bundle, which implies that $\cQ\cong \reg_f(1)$. 
Now, \cite[Prop.\ 4.2]{Kru3} says that $\Phi^{-1}\circ \delta_*\cong \iota_*\bigl(\reg_f(-2)\otimes f^*(\_)\bigr)[1]$. Combining this with \cite[Thm.\ 4.26(i)]{KPScyclic}, we get 
$\Psi\circ \MM_{\alt_2}\circ \delta_*\cong \iota_*\bigl(\reg_f(-1)\otimes f^*(\_)\bigr)[1]$. Taking the left-adjoints on both sides of this isomorphism gives $I\cong K$.
\end{remark}

\begin{remark}\label{rem:jetreconstruction}
 If $n=\binom{m+d}d$, we have $\IG=X$ and $\iota^*E^{[n]}\cong \Jet^m E$. So in this case, our construction only recovers the $m$-jet object $\Jet^m E$, but not $E$ itself. In the case that $X$ is a curve and $E$ is a vector bundle, the $m$-th jet bundle $\Jet^n E$ is recovered in \cite{Biswas-Nagaraj--reconstructioncurves} by a slightly different construction from $E^{[n]}$. The authors of \textit{loc.\ cit.}\ proceed to reconstruct $E$ from $\Jet^n E$ if $g(X)\ge 2$ and $E$ is semi-stable, using the Harder--Narasimhan filtration. Let us remark that, if $X$ is a curve and $L\in \Pic X$ is a line bundle, the exact triangles \eqref{eq:jettriangle} yield the formula
\[
 \det(\Jet^nL)\cong L^{\otimes n+1}\otimes \omega_X^{\otimes \binom{n+1}2}\,.
\]
Hence, if $X=\IP^1$, one can always recover a line bundle $L\in \Pic(\IP^1)$ from its $n$-jet bundle $\Jet^n(L)$ which in turn can be recovered from $L^{[n]}$.
Another way to reconstruct a line bundle $L\in \Pic( \IP^1)$ from $L^{[n]}\in \VB(X^{[n]})$ is to use formulae for the characteristic classes of $L^{[n]}$; see \cite{Mattuck--sym, Wang--tautintegrals}.   
\end{remark}

\subsection{Using a nested Hilbert scheme}\label{sect:nested}
We now turn to the case where $n = \binom{d+m}{m}$ for some $m$.
\begin{lemma}\label{lem:xi}
 Let $\xi\subset X$ be a subscheme concentrated in one point $x\in X$ with $\fm_x^{m+1}\subset \cI_\xi\subset \fm^m_x$ and $\ell(\xi)=\dim_{\IC}(\reg_{X,x}/\fm_x^{m+1})-1=\binom{d+m}d -1$. Let $\cI_{\xi}(x)=\cI_{\xi}/\fm_x$ be the fibre in $x$ of the ideal sheaf of $\xi$. Then $\dim_\IC \cI_{\xi}(x)=1+\binom{d+m}{d-1}-d$.
\end{lemma}

\begin{proof}
We can assume for simplicity that $X=\IA^d$ is the affine space and $x=0$ is the origin.
We write $\cI:=\cI_\xi$ and $\fm:=\fm_x=(x_1,\dots, x_d)$, which are ideals in $\IC[x_1,\dots, x_d]$, and set $V:=\cI(x)=\cI/\fm\cdot \cI$. By assumption, we have $I=\fm^{m+1}+(h)$ for some $h\in \IC[x_1,\dots, x_d]$ homogenous of degree $m$. Hence, $\fm\cdot I=\fm^{m+2}+\fm\cdot (h)$. We consider the subspace $U:=\fm^{m+1}/\fm\cdot I$ of $V$.
Note that the $\IC$-linear map 
\[
\fm/\fm^2\to \fm^{m+1}/\fm^{m+2}\quad,\quad \bar x_i\mapsto \overline{x_i\cdot h} 
\]
is injective. Hence, $\fm\cdot (h)/\fm^{m+2}$ is a $d$-dimensional subspace of the $\binom{d+m}{d-1}$-dimensional vector space $\fm^{m+1}/\fm^{m+2}$. Since $U=\bigl(\fm^{m+1}/\fm^{m+2}\bigr)/\bigl(\fm\cdot (h)/\fm^{m+2}\bigr)$, we get $\dim U= \binom{d+m}{d-1} -d$. Furthermore, the quotient $V/U$ is one-dimensional, spanned by $\bar h$. In summary, 
\[
 \dim V=\dim (U/V) +\dim U= 1+\binom{d+m}{d-1}-d\,.\qedhere
\]
\end{proof}

We now consider the set-up of \autoref{subsec:family} with $\ell=\rank(\Sym^m\Omega_X)-1=\binom{m+d-1}{d-1}$. This means that $f\colon \IG\to X$ is actually a $\IP$-bundle, and the family $\cZ$, that we constructed in \autoref{subsec:family} parametrises punctual subschemes of length $\binom{m+d}d-1$ with the property of $\xi$ of \autoref{lem:xi}.

\begin{lemma}
The non-derived pull-back $\cE:=(\id_{\IG},f)^*\cI_{\cZ}$ is a vector bundle of rank $1+\binom{d+m}{d-1}-d$ on $\IG$. 
\end{lemma}

\begin{proof}
It is sufficient to prove that, for every closed point $t\in \IG$, the fibre $\cE(t)$ is of dimension $1+\binom{d+m}{d-1}-d$; see e.g.\ \cite[Ex.\ 5.8]{HarAGbook}. Let $\xi:=\cZ_t\subset X$ be the fibre of the family $\cZ\subset \IG\times X$ over $t$. Then $\xi$ is a subscheme concentrated in $x:=f(t)$ with the same properties as $\xi$ in \autoref{lem:xi}. 
Considering the commutative diagram of closed embeddings
\[
\xymatrix{
\{t\}\ar^{t\,\mapsto (t,x)\quad}[r]\ar[d] & \{t\}\times X\ar[d]\\
\IG\ar[r] & \IG\times X\,,
}
\]
we see that $\cE(t)\cong \cI_{\cZ\mid \{t\}\times X}(t,x)\cong \cI_{\xi}(x)$, where the last isomorphism uses the flatness of $\cZ$ over $\IG$. Hence, the result follows from \autoref{lem:xi}.
\end{proof}

We now consider the $\IP$-bundle $p\colon Y:=\IP(\cE)\to \IG$ and set $g:=f\circ p\colon Y\to X$. There is the commutative diagram with cartesian squares
\begin{align*}
\xymatrix{
Y \ar@^{(->}^{(\id_{Y},g)\quad}[r]  \ar^{p}[d] \ar@/_6mm/_{g}[dd] & Y\times X\ar^{p\times \id_X}[d] \ar@/^4mm/^{\pr_X}[drr] &  \\
 \IG \ar@^{(->}^{(\id_{\IG},f)\quad}[r]  \ar^{f}[d] & \IG\times X\ar^{f\times \id_X}[d] \ar^{\pr_X}[rr] & & X \\
 X \ar@^{(->}^{\delta\quad}[r] & X\times X \ar@/_4mm/_{\pr_X}[rru] & &.
}
\end{align*}
Let $\alpha\colon p^*\cE\to \reg_p(1)$ be the universal rank 1 quotient. We set $\cJ:=(p\times \id_X)^*\cI_\cZ$, which, by flatness of $p\times \id_X$, is an ideal sheaf on $Y\times X$. Furthermore, we consider the unit of adjunction 
\[
 \eta\colon \cJ\to (\id_Y,g)_*(\id_Y,g)^*\cJ\cong (\id_Y,g)_*p^*(\id_{\IG}, f)^*\cI_{\cZ}\cong (\id_Y,g)_*p^* \cE 
\]
and the composition $\beta:=(\id_Y,g)_*\alpha\circ \eta\colon \cJ\to (\id_Y,g)_*\reg_p(1)$, which is again surjective. We set $\cJ':=\ker(\beta)$ and denote the subscheme corresponding to this ideal sheaf by $\cZ'\subset Y\times X$. Applying the snake lemma to the diagram
\begin{align*}
\xymatrix{
& 0\ar[d] &0\ar[d] & & \\
0\ar[r] & \cJ'\ar[d] \ar[r] & \reg_{Y\times X}\ar^{\id}[d] \ar[r] & \reg_{\cZ'}\ar[d] \ar[r] & 0 \\
0\ar[r] & \cJ\ar^\beta[d]\ar[r] & \reg_{Y\times X} \ar[r]\ar[d] & (p\times \id_X)^*\reg_{\cZ}\ar[d] \ar[r] & 0 \\
 & (\id_Y,g)_*\reg_p(1)\ar[d] &0 &0 & \\
& 0 & & & 
}
\end{align*}
yields a short exact sequence
\begin{align}\label{eq:Z'}
 0\to (\id_Y,g)_*\reg_p(1)\to \reg_{\cZ'}\to (p\times \id_X)^*\reg_\cZ\to 0\,.
\end{align}
It follows that $\cZ'$ is flat of degree $\deg(\cZ)+1=\binom{m+d}d$ over $Y$. We set $n:=\binom{m+d}d$ and denote the classifying morphism for $\cZ'$ by $\nu\colon Y\to X^{[n]}$. 
The morphism $(p, \nu) \colon Y \to \mathbb G \times X^{[n]}$ embeds $Y$ as the subscheme of $\mathbb G \times X^{[n]}$ whose points are pairs
\[
  \{(\xi, \xi') \mid \xi \in \mathbb G, \xi \in X^{[n]}, \xi \subset \xi'\}.
\]

\begin{theorem}\label{thm:Pres}
For $d\ge 2$ and $n=\binom {m+d}d$ for some $m\in \IN$, the functor 
 \[
  N:=g_*\bigl(\reg_p(-1)\otimes \nu^*(\_)\bigr)\colon \D(X^{[n]})\to \D(X)
 \]
is left-inverse to $(\_)^{[n]}$. Also, $N[1]$ is left-inverse to $\FF_n$.
\end{theorem}

\begin{proof}
The proof is very similar to that of \autoref{thm:Gres}.
Let $E\in \D(X)$. By \autoref{lem:tautbasechange}, we have a natural isomorphism $\nu^*E^{[n]}\cong \FM_{\reg_{\cZ'}}(E)$. The short exact sequence \eqref{eq:Z'} induces the exact triangle 
\begin{align}\label{eq:sometriangle}
 \FM_{(\id_Y,g)_*\reg_p(1)}(E)\to \FM_{\reg_{\cZ'}}(E)\to \FM_{(p\times \id_X)^*\reg_{\cZ}}(E)\to \FM_{(\id_Y,g)_*\reg_p(1)}(E)[1]\,. 
\end{align}
By projection formula, we have $\FM_{(\id_Y,g)_*\reg_p(1)}(E)\cong \reg_p(1)\otimes g^*E$. By flat base change, we have $\FM_{(p\times \id_X)^*\reg_{\cZ}}(E)\cong p^*\FM_{\reg_{\cZ}}(E)$. Hence,
\eqref{eq:sometriangle} can be rewritten as 
\[
\reg_p(1)\otimes g^*E\to \nu^*E^{[n]}\to p^*\FM_{\reg_{\cZ}}(E)\to \reg_p(1)\otimes g^*E[1]\,. 
\]
Applying $p_*(\reg_p(-1) \otimes \_)$, we get the exact triangle
\[
p_*p^*f^*E\to p_*(\reg_p(-1) \otimes \nu^*E^{[n]})\to p_*\bigl(\reg_p(-1) \otimes p^*\FM_{\reg_{\cZ}}(E)\bigr)\to \cQ\otimes p_*p^*f^*E[1]\,. 
\]
By \autoref{lem:Grasssod}, we see that $p_*\bigl(\reg_p(-1) \otimes p^*\FM_{\reg_{\cZ}}(E)\bigr)$ vanishes, and we get a natural isomorphism $p_*(\reg_p(-1) \otimes \nu^*E^{[n]})\cong p_*p^*f^*E \cong f^*E$.
Note that the assumption $d\ge 2$ is needed for the above vanishing since, for $d=1$, we would have $\rank(\cE)=1$ so that $p\colon \IP(\cE)\to \IG$ is an isomorphism.
 Applying $f_*$, we get a natural isomorphism 
\[
 N(E^{[n]})\cong f_*p_*(\reg_p(-1) \otimes \nu^*E^{[n]})\cong f_*f^*E\cong E\,,
\]
which means that we have an isomorphism of functors $N\circ (\_)^{[n]}\cong \id$.

Now, in order to prove that $N\circ \FF_n\cong [-1]$, because of \eqref{eq:FFF}, we only need to check that $N(\reg_{X^{[n]}})=0$. As $\nu^*\reg_{X^{[n]}}\cong \reg_Y\cong p^*\reg_{\IG}$, this follows directly from \autoref{lem:Grasssod}(i).  
\end{proof}

\begin{theorem}\label{thm:lefti}
For $d\ge 2$, the functors $(\_)^{[n]}\colon \D(X)\to \D(X^{[n]})$ and  $\FF_n\colon \D(X)\to \D(X^{[n]})$ both have a left-inverse for every $n\in \IN$.
\end{theorem}

\begin{proof}
 If $n=\binom{m+d}d$ for some $m\in \IN$, we get a left-inverse by \autoref{thm:Pres}. If $n$ is not of this form, we get a left-inverse by \autoref{thm:Gres}.
\end{proof}

\section{Reconstruction using fixed-point free automorphisms}\label{sect:freeautos}

\subsection{Multigraphs as families of points}\label{subsect:multigraphs}

Let $\phi_1,\dots,\phi_n\in \Aut(X)$ be automorphisms with \textit{empty pairwise equalisers}, which means that $\phi_i(x)\neq \phi_j(x)$ for every $i\neq j$ and every $x\in X$. Then the graphs $\Gamma_{\phi_i}\subset X\times X$ are pairwise disjoint and 
\[
 \Gamma:=\bigsqcup_{i=1}^n\Gamma_{\phi_i}\subset X\times X
\]
is a family of reduced length $n$ subschemes over $X$. We denote its classifying morphism by \[\psi:=\psi_{\{\phi_1,\dots,\phi_n\}}\colon X\to X^{[n]}\,.\] This morphism maps to the open part $X^{[n]}_0$ of reduced subschemes, which is naturally isomorphic to the open part $X^{(n)}_0$ of the symmetric product. Hence, we can equivalently describe $\psi$ as the morphism $X\to X^{(n)}_0$, $x\mapsto \phi_1(x)+\dots+\phi_n(x)$.
Note that $\psi$ is always finite but, in general, not a closed embedding. For example, if $n=2$, $\phi_1=\id$, and $\phi_2=\iota$ is a fixed-point free involution, we have $\psi(x)=\psi(\iota(x))$ for all $x\in X$. In that case, $\psi$ factorises over a closed embedding $X/\iota\hookrightarrow X^{[2]}$ of the quotient.

\begin{lemma}\label{lem:psipullback}
We have an isomorphism of functors
$\psi^*\circ (\_)^{[n]}\cong \phi_1^*\oplus\phi_2^*\oplus\dots\oplus \phi_n^*$. 
\end{lemma}
\begin{proof}
By \autoref{lem:tautbasechange}, we have $\psi^*\circ (\_)^{[n]}\cong \FM_{\reg_\Gamma}$. Now, the assertion follows from  
the facts that $\reg_{\Gamma}\cong \reg_{\Gamma_{\phi_1}}\oplus \reg_{\Gamma_{\phi_2}}\oplus\dots\oplus \reg_{\Gamma_{\phi_n}}$
and $\FM_{\reg_{\Gamma_{\phi_i}}}\cong\phi_i^*$.  
\end{proof}

\subsection{Reconstruction using multigraphs}

\begin{theorem}\label{thm:fixedfree}
Let $X$ be a smooth projective variety such that there exist a set of $n+1$ automorphisms $\{\phi_0,\phi_1,\dots,\phi_n\}\subset \Aut(X)$ with empty pairwise equalisers. Then the tautological functor $(\_)^{[n]}\colon \D(X)\to \D(X^{[n]})$ is injective on isomorphism classes and faithful. 
\end{theorem}

\begin{cor}\label{cor:abelian}
Let $A$ be an abelian variety. Then the functor $(\_)^{[n]}\colon \D(A)\to \D(A^{[n]})$ is injective on isomorphism classes and faithful for every $n\in \IN$.
\end{cor}

\begin{proof}
There is an infinite subgroup of $\Aut(A)$ whose elements have empty pairwise equalisers, namely the subgroup of translations.    
\end{proof}

Note that \autoref{cor:abelian} also applies to elliptic curves, while all the other reconstruction results presented in this paper require $\dim X$ to be at least 2.

\begin{proof}[Proof of \autoref{thm:fixedfree}]
Replacing $\phi_i$ by $\phi_0^{-1}\circ \phi_i$, we may assume without loss of generality that $\phi_0=\id$.
For $j=0,\dots, n$, let $\psi_j:=\psi_{\{\phi_i\mid i\neq j\}}\colon X\to X^{[n]}$ be the classifying morphism for $\Gamma=\sqcup_{i=0,\dots, n\,,\, i\neq j} \Gamma_{\phi_i}\subset X\times X$, as discussed in \autoref{subsect:multigraphs}. 

By \autoref{lem:psipullback}, the composition $\psi_j\circ (\_)^{[n]}$, for any $j=0,\dots,n$, is faithful. This implies the faithfulness of $(\_)^{[n]}$.

Now, let $E,F\in \D(X)$ with $E^{[n]}\cong F^{[n]}$. Then, by \autoref{lem:psipullback}, we have 
\[
\bigoplus_{i\in\{0,\dots,n\}\,,\, i\neq j}\phi_i^*(E)\cong \psi_j^*E^{[n]}\cong \psi_j^*F^{[n]}\cong \bigoplus_{i\in\{0,\dots,n\}\,,\, i\neq j}\phi_i^*(F) 
\]
for every $j=0,\dots, n$. The category $\D(X)$ is a Krull-Schmidt category; see \cite[Cor.\ B]{Le-Chen--Karoubian}. Hence, $E\cong F$ follows from \autoref{prop:Krull-Schmidt} below.  
\end{proof}

\begin{prop}\label{prop:Krull-Schmidt}
Let $\cD$ be a Krull-Schmidt category, let $\Phi_0,\Phi_1,\dots, \Phi_n\le \Aut(\cD)$ be pairwise distinct linear autoequivalences with $\Phi_0=\id_{\cD}$. Then
for $E,F\in \cD$, we have:
\[
\bigoplus_{i\in\{0,\dots,n\}\,,\, i\neq j}\Phi_i(E)\cong \bigoplus_{i\in\{0,\dots,n\}\,,\, i\neq j}\Phi_i(F) \quad \forall\, j=0,\dots,n  \quad\Longrightarrow \quad E\cong F\,.  
\]
\end{prop}
\begin{proof}
For $j=0,\dots, n$, we set 
\[
\IB_j:=\bigoplus_{i\in\{0,\dots,n\}\,,\, i\neq j}\Phi_i(E)\cong \bigoplus_{i\in\{0,\dots,n\}\,,\, i\neq j}\Phi_i(F)\,.
\]
We note that the number of indecomposable factors (with multiplicity) of every $\IB_j$ is $n$ times the number of irreducible factors of $E$ as well as $n$ times the number of irreducible factors of $F$. In particular, the number of indecomposable factors of $E$ and of $F$ is the same, and we can argue by induction over that number. 

As the base case of the induction we can take the numbers of factors to be zero, in which case $E\cong 0\cong F$. Now assume that $E$ and $F$ both have $k>0$ indecomposable factors. By the cancellation property in Krull-Schmidt categories, we have 
\begin{align}\label{eq:congcondition}
 \IB_0\cong \IB_j\quad\Longleftrightarrow \quad E\cong \Phi_j(E) \text{ and } F\cong \Phi_j(F)\,.
\end{align}
If $\IB_0\not \cong \IB_\ell$ for some $\ell\in\{1,\dots, n\}$, we pick some indecomposable object $B$ that occurs in $\IB_\ell$ with a bigger multiplicity than in $\IB_0$. This $B$ must be an indecomposable factor of $E$ and of $F$. We write $E\cong E' \oplus B$ and $F\cong F'\oplus B$. It follows that, for every $j=0,\dots, n$, 
\[
\bigoplus_{i\in\{0,\dots,n\}\,,\, i\neq j}\Phi_i(E')\cong \bigoplus_{i\in\{0,\dots,n\}\,,\, i\neq j}\Phi_i(F')\cong \IB_j' \quad\text{with}\quad \IB_j\cong\IB_j'\oplus \bigoplus_{i\in\{0,\dots,n\}\,,\, i\neq j}\Phi_i(B) \,. 
\]
Hence, we can apply the induction hypothesis to get $E'\cong F'$, which implies $E\cong F$.

Now, assume that $\IB_0\cong \IB_\ell$ for every $\ell=1,\dots, n$. Then \eqref{eq:congcondition} shows that $E\cong \Phi_\ell(E)$ and $F\cong \Phi_\ell(F)$ for every $\ell$. Hence,
$E^{\oplus n}\cong \IB_0\cong F^{\oplus n}$. In a Krull-Schmidt category, $E^{\oplus n}\cong F^{\oplus n}$ implies $E\cong F$.
\end{proof}

\bibliographystyle{alpha}
\addcontentsline{toc}{chapter}{References}
\bibliography{references}

\end{document}

%% file: header.tex
\usepackage[latin1]{inputenc}
\usepackage[OT1]{fontenc}
\usepackage{amsmath}
\usepackage{amsthm}
\usepackage{amssymb}
\usepackage[all]{xy}
\usepackage{bbm}
\usepackage{amscd}
\usepackage{a4wide}
\usepackage{enumitem}
\usepackage{mathtools}

\setenumerate[0]{label=(\roman*)}

\DeclareMathOperator{\Refl}{\mathsf{Refl}}
\DeclareMathOperator{\Perf}{\mathsf{Perf}}
\DeclareMathOperator{\Spec}{\mathsf{Spec}}
\DeclareMathOperator{\Jet}{\mathsf{Jet}}
\DeclareMathOperator{\pr}{\mathsf{pr}}

\DeclareMathOperator{\id}{\mathsf{id}}
\DeclareMathOperator{\sgn}{\mathsf{sgn}}
\DeclareMathOperator{\Hom}{\mathsf{Hom}}

\DeclareMathOperator{\sHom}{\mathcal{H}\textit{om}}

\DeclareMathOperator{\Coh}{\mathsf{Coh}}
\DeclareMathOperator{\QCoh}{\mathsf{QCoh}}

\DeclareMathOperator{\Ho}{\mathsf H}

\DeclareMathOperator{\rank}{\mathsf{rank}}
\DeclareMathOperator{\supp}{\mathsf{supp}}

\DeclareMathOperator{\Sym}{\mathsf{Sym}}
\DeclareMathOperator{\Gr}{\mathsf{Gr}}

\let\deg\relax
\DeclareMathOperator{\deg}{\mathsf{deg}}

\let\det\relax
\DeclareMathOperator{\det}{\mathsf{det}}

\newcommand{\an}{\mathsf{an}}

\newcommand{\D}{{\mathsf D}}

\DeclareMathOperator{\Aut}{\mathsf{Aut}}

\DeclareMathOperator{\Pic}{\mathsf{Pic}}

\newcommand{\cI}{{\mathcal I}}
\newcommand{\cJ}{{\mathcal J}}

\newcommand{\IA}{\mathbb{A}}
\newcommand{\IC}{\mathbb{C}}

\newcommand{\IN}{\mathbb{N}}
\newcommand{\IP}{\mathbb{P}}

\newcommand{\IB}{\mathbb{B}}
\newcommand{\IG}{\mathbb{G}}

\DeclareMathOperator{\Res}{\mathsf{Res}}
\DeclareMathOperator{\VB}{\mathsf{VB}}
\DeclareMathOperator{\FM}{\mathsf{FM}}

\DeclareMathOperator{\MM}{\mathsf{M}}

\DeclareMathOperator{\FF}{\mathsf{F}}

\DeclareMathOperator{\CC}{\mathsf{C}}
\DeclareMathOperator{\sfS}{\mathsf{S}}

\makeatletter
\newcommand{\leqnomode}{\tagsleft@true}
\newcommand{\reqnomode}{\tagsleft@false}
\makeatother

\let\ker\relax
\DeclareMathOperator{\ker}{\mathsf{ker}}

\let\dim\relax
\DeclareMathOperator{\dim}{\mathsf{dim}}

\newcommand{\sym}{\mathfrak S}

\newcommand{\cF}{\mathcal F}

\newcommand{\cE}{\mathcal E}

\newcommand{\cZ}{\mathcal Z}

\newcommand{\cC}{\mathcal C}
\newcommand{\cD}{\mathcal D}

\newcommand{\cQ}{\mathcal Q}

\newcommand{\fm}{\mathfrak m}
\newcommand{\alt}{\mathfrak a}

\newcommand{\reg}{\mathcal O}

\newcommand{\eps}{\varepsilon}

\renewcommand{\theta}{\vartheta}
\renewcommand{\rho}{\varrho}
\renewcommand{\phi}{\varphi}
\renewcommand{\_}{\underline{\,\,\,\,}}

\usepackage[usenames,dvipsnames]{xcolor}
\usepackage{hyperref}
\hypersetup{
    colorlinks,
    linkcolor={red!50!black},
    citecolor={blue!50!black},
    urlcolor={blue!80!black}
}
\usepackage{aliascnt} 

\newtheorem{theorem}{Theorem}[section]

  \newaliascnt{proposition}{theorem}
  \newtheorem{prop}[proposition]{Proposition}
  \aliascntresetthe{proposition}

  \newaliascnt{lemma}{theorem}
  \newtheorem{lemma}[lemma]{Lemma}
  \aliascntresetthe{lemma}

  \newaliascnt{corollary}{theorem}
  \newtheorem{cor}[corollary]{Corollary}
  \aliascntresetthe{corollary}

\theoremstyle{definition}

  \newaliascnt{definition}{theorem}
  
  \aliascntresetthe{definition}

  \newaliascnt{remark}{theorem}
  \newtheorem{remark}[remark]{Remark}
  \aliascntresetthe{remark}

  \newaliascnt{condition}{theorem}
  
  \aliascntresetthe{condition}

  \newaliascnt{question}{theorem}
  \newtheorem{question}[question]{Question}
  \aliascntresetthe{question}

  \newaliascnt{example}{theorem}
  
  \aliascntresetthe{example}

